\documentclass[11pt, reqno]{amsart}
\usepackage{rsp}

\title[On Integers Whose Sum is the Reverse of their Product]
      {On Integers Whose Sum is the Reverse of their Product}

\author{Xander Faber}
\address{Institute for Defense Analyses \\
Center for Computing Sciences \\
17100 Science Drive \\
Bowie, MD} 
\email{awfaber@super.org}

\author{Jon Grantham}
\address{Institute for Defense Analyses \\
Center for Computing Sciences \\
17100 Science Drive \\
Bowie, MD} 
\email{grantham@super.org}

\begin{document}


\begin{abstract}
We determine all pairs of positive integers $(a,b)$ such that $a+b$ and $a
\times b$ have the same decimal digits in reverse order:
\[
(2,2), (9,9), (3,24), (2,47), (2,497), (2,4997), (2,49997), \ldots
\]
We use deterministic finite automata to describe our approach, which naturally
extends to all other numerical bases. Our automata are a variation on the notion
of Young graphs, which were introduced by Sloane to study ``reverse multiples''.
\end{abstract}
\maketitle


\section{Introduction}
  During a homeschool math lesson, the first author's children made the curious
  observation that $9 + 9 = 18$ and $9 \times 9 = 81$ --- that is, the sum and
  product are the reverse of each other.  A short computer search revealed the
  more interesting examples
\begin{align*}
  2 + 47 = 49 \quad &\text{and} \quad 2\times 47 = 94 \\
  3 + 24 = 27 \quad &\text{and} \quad 3\times 24 = 72.
\end{align*}
Are there other examples of integer pairs $(a,b)$ for which the digits of $a+b$
are the reverse of the digits of $ab$?

To formalize the problem, we say that a base-$\beta$ representation of a
positive integer is in canonical form if it has no leading zero. A pair of
positive integers $a \le b$ will be called a \textbf{reversed sum-product pair
  for the base $\beta$} if the canonical representations of $a+b$ and $ab$ in
base $\beta$ are the reverse of each other. We insist that all numbers be
written in canonical form in order to avoid examples like $a = 15$ and $b = 624$
in base~10, for which $ab = 9360$ and $a + b = 0639$. (Allowing non-canonical
representations is also interesting, but we do not take up that mantle in this
paper.)

\begin{theorem}
The complete list of reversed sum-product pairs in base 10 is
\[
(2,2), (9,9), (3,24), (2,47), (2,497), (2,4997), (2,49997), \ldots
\]
\end{theorem}

Our proof technique is algorithmic in nature. If $a$ is \textit{too large}, then
$a+b$ will have fewer digits than $ab$. Given any choice of \textit{small} $a$
and a guess for the first and last digit of $b$, we give a recursive procedure
for constructing more digits of $b$. This procedure only continues indefinitely
in one case --- the infinite family $(2,47)$, $(2,497), (2,4997), (2,49997),
\ldots$

Whether the sum and product of two numbers have the same digits in reverse order
depends on the choice of numerical base.  Our procedure for describing reversed
sum-product pairs for the base~10 applies to an arbitrary base~$\beta \ge
2$. Consider base~18, where we use the digits $0, 1, 2, \ldots, 9$, $A$, $B$,
$C$, \ldots, $H$. The complete list of reversed sum-product pairs for the
base~18 is
\begin{align}
  \label{eq:rsp18}
  &(2, 2), (H, H), (3, 37), (4, 25), \notag \\
  &(7, 2483D8), (7,2483D9E483D8), (7, 2483D9E483D9E483D8), \ldots \\
  &(B, 1961DC5), (B, 1961DBG461DC5), (B, 1961DBG461DBG461DC5), \ldots \notag
\end{align}

For a given base~$\beta$ and value $a$, the recursive procedure for constructing
digits of $b$ is best described by a deterministic finite automaton (DFA). In
this paper, a DFA is a directed graph with one vertex designated as the
``initial state'', one or more vertices that are ``accepting states'', and edge
labels from some ``alphabet''. Starting at the initial state of a DFA, we can
walk through the graph while writing down the edge labels we pass. If we stop at
an accepting state, then the string of labels we have written is ``accepted'' by
the DFA.
See \cite[\S2.2]{Hopcroft_Ullman} for the formal definition of a DFA and many
more details. For additional connections between automata and number theory, we
recommend \cite{Allouche_Shallit_book} and \cite{RSS_automata}.

\begin{figure}[hb]

\begin{tikzpicture}
  \draw [thick,blue] (0,0) circle (12pt);
  \draw [thick,blue] (2,0) circle (12pt);
  \draw [thick,blue] (2,0) circle (10pt);
  \draw [thick,blue] (2,-2) circle (12pt);
  \draw [thick,blue] (2,-2) circle (10pt);  
  \node at (0,0) {$s_i$};
  \node at (2,0) {$s_{1,1}$};
  \node at (2,-2) {$s_o$};

  \draw [thick,->] (-1,0) -- (-.5,0);
  \draw [thick,->] (.5,0) -- (1.5,0);
  \draw [thick,->] (2,-.5) -- (2,-1.5);
  \draw [thick,->] (0.35, -0.35) -- (1.65,-1.65);
  \draw [thick,->] (2.5,.25) .. controls (3.5,1) and (3.5,-1) .. (2.5,-.25);
  \node at (1,.25) {\footnotesize{$(4,7)$}};
  \node at (2.3,-1) {\footnotesize{$(9)$}};
  \node at (3.65,0) {\footnotesize{$(9,9)$}};
  \node at (0.65,-1.1) {\footnotesize{$(2)$}};
\end{tikzpicture}
  \caption{A deterministic finite automaton for $a = 2$ for the base~10.}
  \label{fig:first_dfa}
\end{figure}
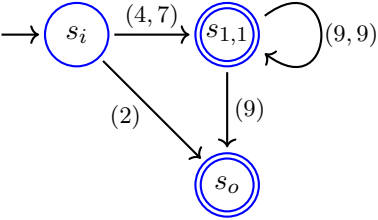

Returning to the issue at hand, integers $b$ that make a base-$\beta$ reversed
sum-product pair with $a$ correspond to accepted strings in a particular
DFA. For example, a DFA for reversed sum-product pairs for the base~10 with $a =
2$ is given in Figure~\ref{fig:first_dfa}.  The initial state is $s_i$. The
accepting states are $s_{1,1}$ and $s_{o}$, drawn with double circles. (The
notation for the states will be explained when we construct the DFAs in
\S\ref{sec:DFA}.)  As we walk through the DFA along directed edges, the edge
labels describe how to build $b$ --- not from left-to-right, but from
out-to-in. Consider the sequence of states $s_i \to s_{1,1} \to s_{1,1}$. The
state $s_{1,1}$ is accepting, so we can legally stop there. The associated
sequence of edge labels is $(4,7), (9, 9)$. The first term tells us that $b =
4\cdots 7$; the second term gives $b = 4997$. Similarly, the sequence of states
$s_i \to s_{1,1} \to s_{1,1} \to s_o$ gives rise to $b = 49997$.

Looking again at our list of reversed sum-product pairs for the base 18 in
\eqref{eq:rsp18}, the patterns become more apparent when we examine the
associated DFAs. For example, Figure~\ref{fig:dfa18} illustrates the DFA for $a
= 7$. Going around the cycle in the DFA gives an infinite family of values $b$
such that $(7,b)$ is a reversed sum-product pair for the base 18.
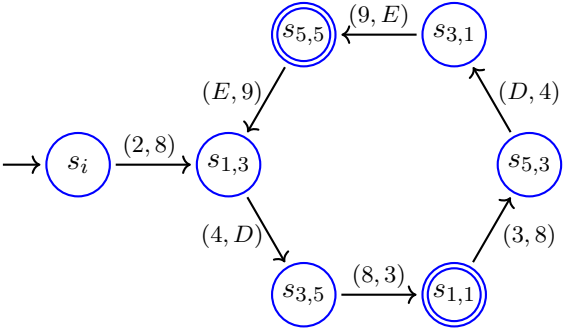
\begin{figure}[htb]

\begin{tikzpicture}
  \draw [thick,blue] (2,0) circle (12pt);     
  \draw [thick,blue] (4,0) circle (12pt);     
  \draw [thick,blue] (5,-1.73) circle (12pt); 
  \draw [thick,blue] (7,-1.73) circle (12pt); 
  \draw [thick,blue] (7,-1.73) circle (10pt);
  \draw [thick,blue] (8,0) circle (12pt);     
  \draw [thick,blue] (7,1.73) circle (12pt);  
  \draw [thick,blue] (5,1.73) circle (12pt);  
  \draw [thick,blue] (5,1.73) circle (10pt);  
  
  \node at (2,0) {$s_i$};
  \node at (4,0) {$s_{1,3}$};
  \node at (5,-1.73) {$s_{3,5}$};
  \node at (7,-1.73) {$s_{1,1}$};
  \node at (8,0) {$s_{5,3}$};
  \node at (7,1.73) {$s_{3,1}$};
  \node at (5,1.73) {$s_{5,5}$};

  \draw [thick,->] (1,0) -- (1.5,0);
  \draw [thick,->] (2.5,0) -- (3.5,0);
  \draw [thick,->] (4.25,-0.433) -- (4.75,-1.299);
  \draw [thick,->] (5.5, -1.73) -- (6.5, -1.73);
  \draw [thick,->] (7.25,-1.297) -- (7.75,-0.433);
  \draw [thick,->] (7.75,.433) -- (7.25,1.299);  
  \draw [thick,->] (6.5, 1.73) -- (5.5, 1.73);    
  \draw [thick,->] (4.75,1.297) -- (4.25,.43);

  \node at (2.95,.25) {\footnotesize{$(2,8)$}};
  \node at (4.05,-.95) {\footnotesize{$(4,D)$}};
  \node at (6,-1.48) {\footnotesize{$(8,3)$}};
  
  \node at (8,-.95) {\footnotesize{$(3,8)$}};
  \node at (8,.95) {\footnotesize{$(D,4)$}};
  
  \node at (6,1.98) {\footnotesize{$(9,E)$}};
  \node at (4.05,.95) {\footnotesize{$(E,9)$}};    
\end{tikzpicture}
  \caption{A deterministic finite automaton for the base~18 with $a = 7$.}
  \label{fig:dfa18}
\end{figure}

The upshot of our investigation for arbitrary bases is the following finiteness
result --- see \S\ref{sec:DFA} for a more precise statement:

\begin{theorem}
  \label{thm:finite}
  Fix a base $\beta \ge 2$. For each $1 \le a < \beta$, there is a deterministic
  finite automaton $A_{\beta,a}$ whose accepted strings correspond to the values
  $b$ such that $(a,b)$ is a reversed sum-product pair for the
  base~$\beta$. Conversely, every reversed sum-product pair arises in this way.
\end{theorem}

A closely related phenomenon can be found in ``reverse multiples'': integers
whose digit reversals are multiples of themselves. For example, in base 10, the
only 4-digit reverse multiples are $9 \times 1089 = 9801$ and $4 \times 2178 =
8712$. Young found a construction of these numbers using special rooted trees in
\cite{young1,young2}, and Sloane reworked these trees into a DFA construction
similar to ours \cite{2178}, though the author refers to them as ``Young
graphs''. See \cite{holt_permutiples} for yet another variation on this theme:
integers $n$ for whom some nontrivial multiple permutes the digits of $n$.

We will break our discussion of the algorithm for describing reversed
sum-product pairs into three sections, corresponding to the relative sizes of
$a$, $b$, and the base $\beta$:
\begin{itemize}
\item (small $b$) $a \le b < \beta$
  \smallskip
\item (large $a$) $\beta < a \le b$
  \smallskip
\item (small $a$, large $b$) $a < \beta < b$
\end{itemize}
In Section~\ref{sec:single-digit}, we show that there are only two reversed
sum-product pairs when $b$ is small: $(2,2)$ and $(\beta-1,\beta-1)$. In
Section~\ref{sec:large-a}, we find there is no reversed sum-product pair with
large~$a$. For the final case, it will be useful to know that we do not need to
``carry'' when computing the sum of $a$ and $b$; this is proved in
Section~\ref{sec:prelims}. We describe our recursive algorithm and the
construction of DFAs in Section~\ref{sec:algorithm}, including a careful
explanation for the base~10. Python code for exploring and visualizing this
construction for an arbitrary base is available at
\begin{center}\url{https://github.com/RationalPoint/reverse}. \end{center}

Next we turn to a kind of opposite problem. Instead of fixing the numerical
base, we fix a positive integer $a$ and ask for which bases $\beta > a$ there
exists a reversed sum-product pair containing $a$. Remarkably, this set has an
enormous amount of structure. Let us say that $(2,2)$ and $(\beta-1,\beta-1)$
are the \textbf{uninteresting} reversed sum-product pairs because they are
present for all but the smallest bases $\beta$; any other pair is
\textbf{interesting}.

\begin{theorem}
  Fix $a \ge 2$. The set of bases $\beta$ for which there exists an interesting
  base-$\beta$ reversed sum-product pair $(a,b)$ is the union of a nonzero
  finite number of arithmetic progressions modulo $a^2 - 1$.
\end{theorem}

In particular, for a fixed $a$, the set of bases $\beta$ for which there exists
an interesting base-$\beta$ reversed sum-product pair containing $a$ has
positive density. We give more precise statements in
Theorem~\ref{thm:participate} and Corollary~\ref{cor:density}, and we calculate
this density for $a \le 10$ at the end of Section~\ref{sec:participate}.

Our investigation led to an intriguing phenomenon that we were unable to fully
explain:

\begin{conjecture}
  \label{conj:interesting}
  The only bases for which there is no interesting reversed sum-product
  pair $(a,b)$ are 
  \[
  2,3,4,5,6,7,8,9,12,15,21.
  \]
\end{conjecture}

Using computer calculation and the tools in Section~\ref{sec:participate}, we
have verified that our conjecture holds for $\beta < 1,441,440$. We also prove
that at least $99.3\%$ of all bases admit an interesting reversed sum-product
pair. This computation is explained in Section~\ref{sec:compute}.

\noindent \textbf{Conventions.}
Throughout this article, we assume that $a \le b$. If an integer $n$ has
base-$\beta$ expansion
\[
n = n_r \beta^r + n_{r-1}\beta^{r-1} + \cdots + n_1\beta + n_0,
\]
we will say that $n_r$ is the ``first digit'' of $n$ and $n_0$ is the ``last
digit''. If $a \le b$ is a reversed sum-product pair for the base $\beta$, then
neither one is divisible by $\beta$; indeed, the product would have trailing
zeros.


\newpage
\section{Small \texorpdfstring{$b$}{b}}
\label{sec:single-digit}

Suppose that $a \le b < \beta$ is a reversed sum-product pair for the base
$\beta$. Then we will show that exactly one of the following is true:
  \begin{itemize}
  \item $(a,b) = (2,2)$ and $\beta \ge 5$; or
    \smallskip
  \item $(a,b) = (\beta-1,\beta-1)$ and $\beta \ge 3$.
  \end{itemize}

  Suppose first that $a+b < \beta$. Then $ab$ and $a+b$ have a single digit in
  base $\beta$, so $ab = a+b$. Rearranging shows
  \[
    (a-1)(b-1) = 1 \quad \Longrightarrow \quad a = b = 2.
  \]
  Our hypothesis that $a + b < \beta$ now becomes $\beta > 4$. 

  Now suppose that $a+b > \beta$. Since  $a+b < 2\beta$, we can write
  \begin{equation}
    \label{eq:one_digit1}
    a + b = \beta + x, \quad 1 \le x < \beta. 
  \end{equation}
  Since $ab$ is the reverse of $a+b$, we have 
  \begin{equation}
    \label{eq:one_digit2}
    ab = x\beta + 1. 
  \end{equation}
  Solving \eqref{eq:one_digit1} for $x$ and inserting into \eqref{eq:one_digit2} gives
  \[
  ab = 1 + a\beta + b\beta - \beta^2.
  \]
  Rearranging yields
  \[
    (\beta-a)(\beta-b) = 1.
  \]
  Since $1 \le (\beta-a), (\beta-b) < \beta$, we must have $a = b =
  \beta-1$. Note that this means $x = \beta-2$, which is a valid first digit
  only when $\beta \ge 3$.


\section{Large \texorpdfstring{$a$}{a}}
  \label{sec:large-a}

If $a$ is large, then we expect $ab$ to have more digits than $a+b$. The
next lemma uses this idea to produce a coarse upper bound for $a$.

\begin{lemma}
  \label{lem:dda}
  Suppose that $a \le b$ is a reversed sum-product pair with $a > \beta$. Then
  $a < 2\beta$ and $b < \beta(\beta+1)$.
\end{lemma}

\begin{proof}
Note that $\beta(a+b)$ has one more digit than $a+b$ in base $\beta$. The fact
that $a+b$ and $ab$ have the same number of digits implies that
$ab<\beta(a+b)$. Solving for $b$ gives
\[
b<\frac{a \beta}{a-\beta}.
\]
The right side is a decreasing function of $a$, so it is maximized when
$a=\beta+1$, which gives the inequality $b<\beta(\beta+1)$. Since $a \le b <
\frac{a \beta}{a-\beta}$, we can solve for $a$ to get $a < 2\beta$.
\end{proof}

We now refine the bound in the lemma and conclude there is no reversed-sum
product pair with large $a$.

\begin{proposition}
    Suppose that $a \le b$ is a reversed sum-product pair for the base
    $\beta$. Then $a < \beta$.
\end{proposition}

\begin{proof}
  Suppose for the sake of a contradiction that $a > \beta$. By
  Lemma~\ref{lem:dda}, we have $a < 2\beta$ and $b < \beta(\beta+1)$. For $\beta
  \le 5$, we can examine all $a \in (\beta,2\beta)$ and $b \in
      [a,\beta(\beta+1))$ and find there is no reversed sum-product pair with
        these constraints.

  For the remainder of the proof, we may assume that
  \[
    \beta < a < 2\beta, \quad a \le b < \beta(\beta+1), \quad \beta \ge 6. 
    \]
  With these assumptions, we find that $a+b < 3\beta + \beta^2$, so that $a+b$
  has 2 or 3 digits.  We write
  \begin{equation}
    \label{eq:a_and_b}
    a = \beta + a_0 \qquad \text{and} \qquad b = b_2 \beta^2 + b_1 \beta + b_0,
  \end{equation}
    where $0 \le a_i,b_i < \beta$ and $a_0b_0 \ne 0$. Note that if $b_2 \ne 0$,
    then $ab$ has 4 digits. So we may further assume that $b_2 = 0$. As $a \le
    b$, it follows that $b_1 \ge 1$.

    \medskip
    
  \noindent \textbf{Case $a+b$ has 2 digits.} From \eqref{eq:a_and_b}, we have
  \[
  a+b = (1+b_1)\beta + (a_0 + b_0) ,
  \]
  so the first digit of $a+b$ is at least~2. Thus,
  \[
  \frac{\beta}{2}(a+b) \ge \beta^2 > ab,
  \]
  since $ab$ must also have $2$ digits. Solving for $a$ gives
  \[
   a < \frac{b\beta}{2b - \beta}.
   \]
  For $b > \beta$ the right side is a decreasing function, so it is maximized
  by taking $b = \beta + 1$. We then have
  \[
    a < \frac{b\beta}{2b - \beta} \le \beta \frac{\beta+1}{\beta+2} < \beta,
    \]
    a contradiction.

    \medskip    
   
  \noindent \textbf{Case $a+b$ has 3 digits.} From \eqref{eq:a_and_b}, we have
  \[
      a + b = (1+b_1)\beta + (a_0 + b_0). 
  \]
  As we are assuming $a+b$ has 3 digits, we have two subcases to consider:
  \begin{itemize}
  \item[(I)] $1+b_1 \ge \beta$, or
    \smallskip
  \item[(II)] $1+b_1 = \beta - 1$ and $a_0 + b_0 \ge \beta$.
  \end{itemize}
  In both (I) and (II), \eqref{eq:a_and_b} shows that the product $ab$ satisfies
  \begin{equation}
    \label{eq:3digits}
    ab = b_1 \beta^2 + (a_0b_1 + b_0)\beta  + a_0b_0  
  \end{equation}

  In case (I), we must have $b_1 = \beta - 1$. Using the fact that $a_0,b_0 \ge
  1$, we obtain the following estimate from \eqref{eq:3digits}:
  \[
    ab \ge \beta^3 + 1.
    \]
  But then $ab$ has at least 4 digits, a contradiction

  In case (II), we look at the coefficient on $\beta$ in \eqref{eq:3digits}:
  \[
    a_0b_1 + b_0 = a_0(\beta-2) + b_0 \ge a_0(\beta - 2) + \beta-a_0 =
    a_0(\beta-3) + \beta.
    \]
  This is an increasing function of $a_0$. If $a_0 \ge 2$, then this quantity is
  at least $2\beta$ since $\beta \ge 6$. As in case (I), we obtain the absurd
  conclusion that $ab$ has 4 digits. So $a_0 = 1$ and $b_0 = \beta - 1$. This
  completely nails down $a$ and $b$:
  \[
    a = \beta + 1 \qquad \text{and} \qquad b = (\beta-2)\beta + (\beta-1). 
    \]
  As $a + b = \beta^2$, we do not obtain a reversed sum-product pair. This
  completes the proof. 
\end{proof}


\section{Carries}
\label{sec:prelims}

From grade school arithmetic, we know about ``carrying'' when computing
multi-digit sums and products. Our primary goal for this section is to show that
there is no carry when computing the sum $a+b$ for a reversed sum-product pair
$(a,b)$. (Typically there are carries in the product.) We begin with a careful
definition of carry digits and a bound on how big they can be.

Suppose that $a, d$ are two single-digit numbers, which means $0 \le a,d <
\beta$. Their sum or product is at most two digits. If it is two digits, we call
the leading digit the \textbf{carry digit}. Upon adding single-digit numbers,
the resulting carry digit is at most~1: $a + d < 2\beta$. If $a \ge 1$, then
multiplying these two numbers produces a carry digit that is strictly less than
$a$:
\[
ad \le a(\beta-1) = (a-1)\beta + (\beta - a). 
\]
Now imagine that we are multiplying a single-digit number $a$ by a multi-digit
number $b$. A particular digit of the product $ab$ comes from multiplying $a$ by
a single digit of $b$ and adding the previous carry digit. The new
carry digit is still at most $a-1$:
\[
ad + \text{previous carry} \le a(\beta-1) + (a-1) = (a-1)\beta + (\beta-1).
\]

Next we show that the first digit of $a+b$ does not arise from a carry.

\begin{lemma}
  \label{lem:same_length}
If $(a,b)$ is a reversed sum-product pair for the base $\beta$ with $a < \beta <
b$, then $b$ and $a+b$ have the same number of digits.
\end{lemma}

\begin{proof}
Assume the number of digits differs. Then $a+b$ has one more digit than
$b$, and $b = \beta^\ell - c$ for some $\ell\ge 2$ and $c$ with $a > c \ge 1$.
Then
\[
a+b=\beta^\ell+(a-c) \qquad \text{and} \qquad ab = a \beta^\ell - ac.
\]

The above expression for $a+b$ has first digit $1$, last digit $a-c$, and all
other digits $0$. Therefore, the reverse is true for $ab$:
\[
ab=(a-c)\beta^\ell+1.
\]
Combining these two expressions for $ab$ and rearranging, we get
$c(\beta^\ell-a)=1$. Since $\beta^\ell\ge\beta^2>a+1$, this is a
contradiction.
\end{proof}


Now we improve the preceding lemma to show that the computation of $a+b$
involves no carry at all.

\begin{proposition}
  \label{prop:no_carry}
  If $(a,b)$ is a reversed sum-product pair for the base~$\beta$ with $a < \beta
  < b$, the last digit of $b$ is strictly smaller than $\beta - a$.
\end{proposition}

\begin{proof}
  Write $b_r$ for the first digit of $b$, and let $d$ be the first digit of
  $ab$. Since $ab$ has the same number of digits as $a+b$, which has the same
  number of digits as $b$ (Lemma~\ref{lem:same_length}), we see that $d = a b_r
  + \lambda < \beta$, where $\lambda \le a-1$ is the carry from the $(r-1)$-st place of
  the product $ab$. It follows that $d \ge a b_r \ge a$.

  Write $b_0$ for the last digit of $b$. Since $a \le b$ is a reversed
  sum-product pair, the last digit of $a+b$ must be $d \equiv a + b_0
  \pmod{\beta}$. That is, $b_0 \equiv d - a \pmod{\beta}$. Since $d \ge a$, we
  conclude that $b_0 = d - a < \beta - a$.
\end{proof}

In particular, the above proposition shows that $a+b$ and $b$ have the same
first digit, a fact we will capitalize on in the next section.






\section{\texorpdfstring{Small $a$, Large $b$}{Small a, Large b}}
\label{sec:algorithm}

Suppose that $(a,b)$ is a reversed sum-product pair for the base $\beta$, and that $a
< \beta < b$. Write the base-$\beta$ expansion of $b$ as
\[
b = b_r \beta^r + \cdots + b_0.
\]
We begin by determining necessary --- though not sufficient --- conditions on
$b_0, b_r$. From there, we will inductively determine 2 more digits (which may
be the same digit if $b$ has an odd number of digits), and so on. With careful
bookkeeping, this procedure will result in only finitely many states, and we
will be able to develop an algorithm for finding all valid reversed sum-product
pairs.


\subsection{The Recursion}
\label{sec:recursion}

To recap, we have now determined that if $(a,b)$ is a reversed sum-product
pair for the base $\beta$ with $a < \beta < b$, then
\begin{itemize}
\item $b$ and $a+b$ and $ab$ have the same number of digits;
  \smallskip
\item $b$ and $a+b$ have the same first digit; and
  \smallskip
\item the last digit of $b$ is strictly smaller than $\beta -a$.
\end{itemize}
Recall that we write $b_r, b_0$ for the first and last digits of $b$,
respectively.  Then we have $0 < b_0 < \beta - a$, and the last digit of $a+b$
is $a + b_0$.  So the first digit of $ab$ is $ab_r + \lambda = a + b_0$ for some
$0 \le \lambda < a$ corresponding to the carry from the $(r-1)$-st place of the
product. (Here, $\lambda$ stands for ``left'' carry.)  The first digit of $a+b$
agrees with $b_r$. But this is also equal to the last digit of $ab$, so we have
$b_r \equiv ab_0 \pmod{\beta}$. Writing $\rho$ for the carry from the units
place --- the ``right'' carry --- we obtain the following constraints on the
first and last digits of $b$:
\begin{align}
  \label{eq:recursion_start}
a + b_0 &= ab_r + \lambda   \text{ for some $0 \le \lambda < a$} \\
   ab_0 &= b_r + \rho \beta \text{ for some $0 \le \rho < a$} \notag
\end{align}

Now suppose that we have determined the first and last $n$ digits of $b$ for
some $n > 0$. Write $b = \cdots xx' \cdots y'y \cdots$, where $x,y$ have already
been determined and we would like to find $x'$ and $y'$. Assume further that we
already know the carry into the $x$-column of the product $ab$ --- let us call
it $\lambda$. (This will be part of our inductive information.) The product of
$a$ with the rightmost $n$ digits of $b$ determines a carry out of the
$y$-column of $ab$ --- call it $\rho$, so that the product will
have $ay' + \rho \pmod{\beta}$ in the next place to the left. See
Figure~\ref{fig:carries}(i). Since $(a,b)$ is a reversed sum-product pair, and
since the corresponding digit of $a+b$ is $x'$, we find that
\[
ay' + \rho =  x' + \rho' \beta \text{ for some $0 \le \rho' < a$}.
\]
To determine the digit of the product $ab$ arising from multiplication by $x'$,
we need to consider the unknown carry from the middle digits --- call it
$\lambda'$. See Figure~\ref{fig:carries}(ii). The result is $ax' + \lambda'$,
which must agree modulo $\beta$ with the corresponding digit of $a+b$, namely
$y'$. That is, 
\[
ax' + \lambda' = y' + \lambda \beta \text{ for some $0 \le \lambda' < a$}.
\]
We combine the recursion equations for future reference:
\begin{align}
  \label{eq:recursion}
    ax' + \lambda' &= y' + \lambda \beta \text{ for some $0 \le \lambda' < a$}\\
    ay' + \rho &= x' + \rho' \beta \text{ for some $0 \le \rho' < a$} \notag
\end{align}
\begin{figure}[htb]
  \begin{tikzpicture}
  \node at (0,0) {$\cdots \ x \ \cdots \ y \ \cdots$};
  \node at (-0.4,.5) {$\stackrel{\lambda}{\curvearrowleft}$};
  \node at (0.3,.5) {$\stackrel{\rho}{\curvearrowleft}$};  
  \node at (0,-0.75) {(i)};
  
  \node at (5,0) {$\cdots \ x\ x' \ \cdots \ y' \ y \ \cdots$};
  \node at (4.15,.5) {$\stackrel{\lambda}{\curvearrowleft}$};
  \node at (4.6,.5) {$\stackrel{\lambda'}{\curvearrowleft}$};  
  \node at (5.8,.5) {$\stackrel{\rho}{\curvearrowleft}$};
  \node at (5.35,.5) {$\stackrel{\rho'}{\curvearrowleft}$};    
  \node at (5,-0.75) {(ii)};  
\end{tikzpicture}
  \caption{An illustration of the carries involved in the beginning and end of
    the recursion step.}
  \label{fig:carries}
\end{figure}

There are two ways for this construction to terminate: when the left and right
sides may be concatenated, or when the left and right sides overlap in a
digit. Suppose we know the first and last $n$ digits of $b$ for some $n \ge
1$. Write $b = \cdots x \cdots y \cdots$, where $x,y$ have already been
determined, and suppose that we know the carry $\lambda$ into the $x$-column of
the product $ab$ and the carry $\rho$ out of the $y$-column in the product. The
left and right sides may be concatenated if the carries are compatible --- i.e.,
if $\lambda = \rho$. In that case, $a$ and $b = \cdots xy \cdots$ are a reversed
sum-product pair.

To understand when the left and right sides may overlap in a digit, we must run
the recursion one more step. With the setup of the previous paragraph, we solve
the recursion equations \eqref{eq:recursion} to obtain
$x',y',\lambda',\rho'$. If $x' = y'$, then we claim that $\lambda' = \rho$ and
$\lambda = \rho'$. To see it, set $y' = x'$ in \eqref{eq:recursion} and subtract
the two equations. We obtain
\[
 \lambda' - \rho = (\lambda - \rho')\beta.
 \]
Since $|\lambda' - \rho| < a < \beta$, we must have $\lambda' = \rho$ and
$\lambda = \rho'$, as desired. It follows that the carries match up so that $a$
and $b = \cdots xx'y\cdots$ are a reversed sum-product pair.

To summarize, we have shown that the above procedure can terminate in two ways:
\begin{itemize}
\item if $\lambda = \rho$ at any step, or
   \smallskip
\item if $x' = y'$ in the recursion step. 
\end{itemize}


\subsection{\texorpdfstring{The Case $\beta = 10$}{The Case base = 10}}
\label{sec:base10}

Using the strategy from the preceding section, we complete the promised
description of all reversed sum-product pairs for the base~10:

\begin{theorem}
  If $a \le b$ is a reversed sum-product pair for the base~10, then $(a,b)$ is
  among
\[
(2,2), (9,9), (3,24), (2,47), (2,497), (2,4997), (2,49997), \ldots
\]  
\end{theorem}

Section~\ref{sec:single-digit} shows that $(2,2)$ and $(9,9)$ are the only
instances with $b < 10$. Section~\ref{sec:large-a} shows that any remaining
pair must have $a < 10 < b$.

For a given $a < 10$, we begin by looking for all 4-tuples
$(b_r,b_0,\lambda,\rho)$ satisfying \eqref{eq:recursion_start} with $0 < b_0 <
10 - a$ and $0 < b_r < 10$.
The result is given in Table~\ref{tab:recurse}.

\begin{table}[!ht]
\begin{tabular}{ccccc}\toprule
  $a$ & $b_r$ & $b_0$ & $\lambda$ & $\rho$ \\
  \midrule
  2 & 4 & 7 & 1 & 1 \\
  3 & 2 & 4 & 1 & 1 \\
  \bottomrule
\end{tabular}
\vspace{.2cm}
\caption{The solutions to \eqref{eq:recursion_start} for $\beta = 10$ with $0<b_r<10$ and $0 < b_0 <
10 - a$.}
\label{tab:recurse}
\end{table}

Let us look at $a = 3$ first. Any reversed sum-product pair $(a,b)$ must have $b
= 2 \cdots 4$ for some unknown (and possibly nonexistent) digits between the 2
and the 4.  The equations \eqref{eq:recursion} defining the recursion step have
no solution, as one can check with a short calculation. Since $\lambda = \rho =
1$ in the first step, we obtain $b = 24$ as the only solution with $a = 3$.


Next we look at $a = 2$. Any reversed sum-product pair $(a,b)$ must have $b = 4
\cdots 7$. Since $\lambda = \rho = 1$, we obtain a first solution $b = 47$. The
recursion equations \eqref{eq:recursion} have a single solution:
$(x',y',\lambda',\rho') = (9,9,1,1)$. Since $x' = y' = 9$, we obtain the value
$b = 497$. Since $\lambda' = \rho'$, we also obtain the value $b =
4997$. Finally, note that the recursion equations depend only on $\lambda,
\rho$; it follows that $(9,9,1,1)$ is the unique solution obtained by running
the recursion again. We obtain the solutions $49997$ and $499997$ from the next
step of the recursion, and so on. This completes the proof of the theorem.


\subsection{A Deterministic Finite Automaton}
\label{sec:DFA}
Fix a base $\beta \ge 2$ and an integer $1 \le a < \beta$. The procedure
described in Section~\ref{sec:recursion} suggests a method for constructing a
DFA, which we now carry out.  The allowable edge labels --- i.e., the
  ``alphabet'' in DFA theory --- are pairs $(x,y)$ with $0 \le x,y < \beta$ as well
as singletons $(x)$ for $0 \le x < \beta$. We construct three types of states:
\begin{itemize}
\item An initial state $s_i$;
  \smallskip
\item An ``odd state'' $s_o$; and
  \smallskip
\item A ``carry state'' $s_{\lambda,\rho}$ for each pair of integers $0 \le \lambda,\rho < a$. 
\end{itemize}
The accepting states are $\{s_o\} \cup \{s_{\lambda,\lambda} \ : \ 0 \le \lambda <
a\}$. The transitions for our DFA are as follows: 
\begin{itemize}
  
  \item For each solution $(b_r,b_0,\lambda,\rho)$ to
    \eqref{eq:recursion_start}, we have a transition from the initial state
    $s_i$ to the state $s_{\lambda,\rho}$ with label $(b_r,b_0)$.

  \smallskip
    
  \item For each carry state $s_{\lambda,\rho}$ and each solution
    $(x',y',\lambda',\rho')$ to the recursion equations \eqref{eq:recursion}, we
    have a transition from $s_{\lambda,\rho}$ to $s_{\lambda',\rho'}$ with label
    $(x',y')$. If $x' = y'$, we also include a transition from
    $s_{\lambda,\rho}$ to the odd state $s_o$ with label $(x')$.

  \smallskip    

  \item If $\beta \ge 5$ and $a = 2$, include a transition from $s_i$ to $s_o$
    with label $(2)$.

  \smallskip    

  \item If $\beta \ge 3$ and $a = \beta - 1$, include a transition from $s_i$ to
    $s_o$ with label $(\beta-1)$.
\end{itemize}
Write $A = A_{\beta,a}$ for the DFA thus constructed.

The definition of our DFA captures the digit construction involved in the
recursion in Section~\ref{sec:algorithm}. We formalize this in the following
statement, which is a more precise version of Theorem~\ref{thm:finite} from the
introduction.

\begin{theorem}
  For $\beta \ge 2$ and $a < \beta$, let $A = A_{\beta,a}$ be the DFA
  constructed above.
  \begin{itemize}
    \item Suppose that the string $(x_1,y_1)(x_2,y_2) \cdots (x_n,y_n)$ is
      accepted by $A$ for some $n \ge 1$. Then $a$ and $b = x_1x_2 \cdots
      x_ny_n\cdots y_2y_1$ are a reversed sum-product pair for the base $\beta$.

      \smallskip
      
    \item Suppose that the string  $(x_1,y_1)(x_2,y_2) \cdots (x_n,y_n)(z)$ is accepted by
      $A$ for some $n \ge 0$. Then $a$ and $b = x_1x_2 \cdots x_nzy_n\cdots
      y_2y_1$ are a reversed sum-product pair for the base $\beta$.
  \end{itemize}
  If $a \le b$ is a reversed sum-product pair for the base $\beta$, then $b$ can
  be constructed from a string accepted by $A$ in one of these two ways.
\end{theorem}

\textit{A priori}, the number of states in $A_{\beta,a}$ is $a^2 + 2$. In
practice, many of these states are unreachable by a path beginning at the
initial state. To avoid these superfluous states, we will always ``lazy
construct'' the DFA: beginning with solutions to \eqref{eq:recursion_start},
only construct states as needed to satisfy the recursion. Taking this approach
does not affect the set of strings accepted by the DFA. For example, if
\eqref{eq:recursion_start} has no solution, then only the initial state needs to
be constructed. See Figure~\ref{fig:base10DFAs} for the end result in base~10.

\begin{figure}[htb]
  \begin{center}

\begin{tikzpicture}
  \node at (1,-2.75) {$a = 2$};

  \draw [thick,blue] (0,0) circle (12pt);
  \draw [thick,blue] (2,0) circle (12pt);
  \draw [thick,blue] (2,0) circle (10pt);
  \draw [thick,blue] (2,-2) circle (12pt);
  \draw [thick,blue] (2,-2) circle (10pt);  
  \node at (0,0) {$s_i$};
  \node at (2,0) {$s_{1,1}$};
  \node at (2,-2) {$s_o$};

  \draw [thick,->] (-1,0) -- (-.5,0);
  \draw [thick,->] (.5,0) -- (1.5,0);
  \draw [thick,->] (2,-.5) -- (2,-1.5);
  \draw [thick,->] (0.35, -0.35) -- (1.65,-1.65);
  \draw [thick,->] (2.5,.25) .. controls (3.5,1) and (3.5,-1) .. (2.5,-.25);
  \node at (1,.25) {\footnotesize{$(4,7)$}};
  \node at (2.3,-1) {\footnotesize{$(9)$}};
  \node at (3.65,0) {\footnotesize{$(9,9)$}};
  \node at (0.65,-1.1) {\footnotesize{$(2)$}};

  
  \node at (7,-2.75) {$a = 3$};

  \draw [thick,blue] (6,-1) circle (12pt);
  \draw [thick,blue] (8,-1) circle (12pt);
  \draw [thick,blue] (8,-1) circle (10pt);
  \node at (6,-1) {$s_i$};
  \node at (8,-1) {$s_{1,1}$};  

  \draw [thick,->] (5,-1) -- (5.5,-1);
  \draw [thick,->] (6.5,-1) -- (7.5,-1);
  \node at (7,-.75) {\footnotesize{$(2,4)$}};  

  
  \node at (1,-5.25) {$a = 9$};

  \draw [thick,blue] (0,-4.5) circle (12pt);
  \draw [thick,blue] (2,-4.5) circle (12pt);
  \draw [thick,blue] (2,-4.5) circle (10pt);
  \node at (0,-4.5) {$s_i$};
  \node at (2,-4.5) {$s_{o}$};  

  \draw [thick,->] (-1,-4.5) -- (-.5,-4.5);
  \draw [thick,->] (.5,-4.5) -- (1.5,-4.5);
  \node at (1,-4.25) {\footnotesize{$(9)$}};

  
  \node at (7,-5.25) {$a = 1, 4, 5, 6, 7, 8$};

  \draw [thick,blue] (7,-4.5) circle (12pt);
  \node at (7,-4.5) {$s_i$};  

  \draw [thick,->] (6,-4.5) -- (6.5,-4.5);
  
\end{tikzpicture}
  \end{center}
  \caption{``Lazy constructed'' DFAs for the base~10.}
  \label{fig:base10DFAs}
\end{figure}
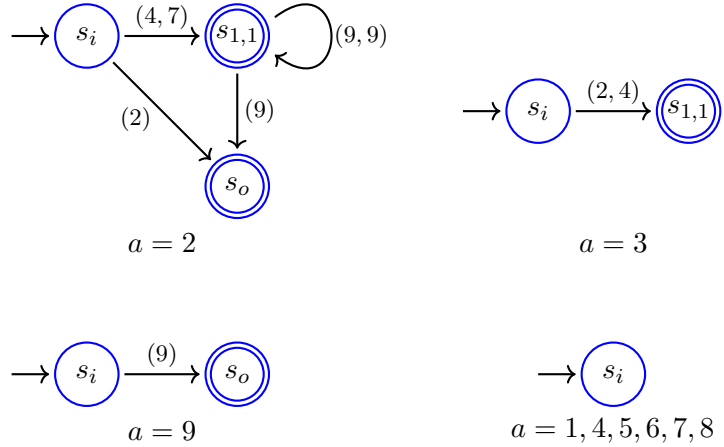

After constructing the DFA, we may find that there are states $s$ that do not
participate in any accepted string --- i.e., no path from the initial state to
an accepting state passes through $s$. We can ``trim'' all such $s$ from the
DFA. As an extreme example, the DFA $A_{150,31}$ has $13$ states accessible from
the initial state, but it has no accepting state. Consequently, we can trim all
states but the initial one. Said another way, there is no base-150 reversed
sum-product pair of the form $(31,b)$.

When constructing examples of these DFAs, one begins to marvel at how
complicated they can be. One way to measure this complexity is by the growth
of the number of strings accepted by $A_{\beta,a}$. Recall that the set of
strings accepted by a DFA is a regular language. Passing to a more general
setting for a moment, let $\Sigma$ be an alphabet and $R \subset \Sigma^*$ any
regular language. Define the growth function
\[
  p_R(n) = |R \cap \Sigma^n| = \text{ ``number of strings in $R$ of length $n$''}.
\]
In \cite{Szilard_et_al}, it is shown that $p_R(n) = O(n^k)$ for some $k \ge 0$
if and only if $R$ can be represented as a finite union of regular expressions
of the form $xy_1^*z_1\cdots y_k^* y_k$, where $x, y_i, z_i$ are fixed strings
in $\Sigma^*$. In essence, this means there are only finitely many repeated
patterns that can occur in the strings of $R$.

Figures~\ref{fig:first_dfa} and~\ref{fig:dfa18} are examples of DFAs with
$p_R(n) = O(1)$. We also have an example with exponential growth. Use the digits
$0, 1, 2, \ldots, 0, A, B, C, \ldots, Q$ to represent integers in
base~$27$. Running the DFA construction in Section~\ref{sec:DFA}, one verifies
that a state diagram for $A_{27,A}$ is given by Figure~\ref{fig:dfa27}. If $R$
is the regular language associated with this DFA, then a straightforward
inductive argument gives the number of accepted strings of a given length:
  \[
  p_R(4+5n) = 2^n \text{ for each $n \ge 0$}.
  \]
In particular, $p_R(n) \ne O(n^k)$  for any $k \ge 0$. 

\begin{figure}[htb]

\begin{tikzpicture}

\draw[thick,blue] (0.00,0.00) circle (12pt);  
\node at (0.00,0.00) {$s_i$};
\draw[thick,blue] (2.50,0.00) circle (12pt);  
\node at (2.50,0.00) {$s_{{1,4}}$};
\draw[thick,blue] (2.50,2.00) circle (12pt);  
\node at (2.50,2.00) {$s_{{5,3}}$};
\draw[thick,blue] (4.50,2.00) circle (12pt);  
\node at (4.50,2.00) {$s_{{6,0}}$};
\draw[thick,blue] (7.50,2.00) circle (12pt);  
\node at (7.50,2.00) {$s_{{0,6}}$};
\draw[thick,blue] (9.50,2.00) circle (12pt);  
\node at (9.50,2.00) {$s_{{3,5}}$};
\draw[thick,blue] (2.50,-2.00) circle (12pt);  
\node at (2.50,-2.00) {$s_{{7,8}}$};
\draw[thick,blue] (4.50,-2.00) circle (12pt);  
\node at (4.50,-2.00) {$s_{{9,2}}$};
\draw[thick,blue] (7.50,-2.00) circle (12pt);  
\node at (7.50,-2.00) {$s_{{2,9}}$};
\draw[thick,blue] (9.50,-2.00) circle (12pt);  
\node at (9.50,-2.00) {$s_{{8,7}}$};
\draw[thick,blue] (9.50,0.00) circle (12pt);  
\node at (9.50,0.00) {$s_{{4,1}}$};
\draw[thick,blue] (6.00,0.00) circle (12pt);  
\draw[thick,blue] (6.00,0.00) circle (10pt);  
\node at (6.00,0.00) {$s_o$};

\draw [thick,->] (-1.00,0.00) -- (-0.50,0.00); 
\draw [thick,->] (0.50,0.00) -- (2.00,0.00); 
\draw [thick,->] (3.00,2.00) -- (4.00,2.00); 
\draw [thick,->] (5.00,2.00) -- (7.00,2.00); 
\draw [thick,->] (8.00,2.00) -- (9.00,2.00); 
\draw [thick,->] (9.00,-2.00) -- (8.00,-2.00); 
\draw [thick,->] (7.00,-2.00) -- (5.00,-2.00); 
\draw [thick,->] (4.00,-2.00) -- (3.00,-2.00); 
\draw [thick,->] (7.50,-1.50) -- (7.50,1.50); 
\draw [thick,->] (4.50,1.50) -- (4.50,-1.50); 
\draw [thick,->] (2.50,0.50) -- (2.50,1.50); 
\draw [thick,->] (2.50,-1.50) -- (2.50,-0.50); 
\draw [thick,->] (9.50,1.50) -- (9.50,0.50); 
\draw [thick,->] (9.50,-0.50) -- (9.50,-1.50); 
\draw [thick,->] (7.20,-1.60) -- (6.30,-0.40); 
\draw [thick,->] (4.80,1.60) -- (5.70,0.40); 

\node at (-0.75,0.00) {\footnotesize{$$}};   
\node at (1.25,0.25) {\footnotesize{$(2,B)$}};   
\node at (3.50,2.25) {\footnotesize{$(D,1)$}};   
\node at (6.00,2.25) {\footnotesize{$(I,I)$}};   
\node at (8.50,2.25) {\footnotesize{$(1,D)$}};   
\node at (8.50,-2.25) {\footnotesize{$(O,Q)$}};   
\node at (6.00,-2.25) {\footnotesize{$(5,5)$}};   
\node at (3.50,-2.25) {\footnotesize{$(Q,O)$}};   
\node at (8.05,0.00) {\footnotesize{$(7,G)$}};   
\node at (4.00,0.00) {\footnotesize{$(G,7)$}};   
\node at (2.05,1.00) {\footnotesize{$(3,8)$}};   
\node at (1.95,-1.00) {\footnotesize{$(K,C)$}};   
\node at (9.95,1.00) {\footnotesize{$(8,3)$}};   
\node at (10.05,-1.00) {\footnotesize{$(C,K)$}};   
\node at (6.5,-1.25) {\footnotesize{$(5)$}};   
\node at (5.5,1.25) {\footnotesize{$(I)$}};   

\end{tikzpicture}
  \caption{A state diagram for the DFA $A_{27,A}$.}
  \label{fig:dfa27}
\end{figure}
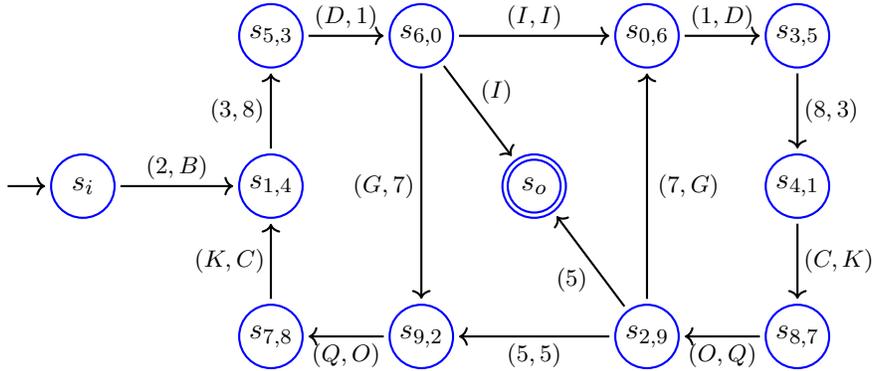

\begin{question}
  Is there a combinatorial description of the set of bases $\beta$ and integers
  $a < \beta$ such that $A_{\beta,a}$ has bounded growth function: $p_R(n) =
  O(1)$?
\end{question}

To close this section, we give an easy criterion for showing that $A_{\beta,a}$
has no accepting state. 

\newpage
\begin{proposition}
  Fix $\beta \ge 2$. If $a \le b$ is a reversed sum-product pair for the base
  $\beta$, then
  \[
  \gcd(\beta-1, a-1) = \gcd(\beta-1,b-1) = 1.
  \]
  In particular, the DFA $A_{\beta,a}$ has no accepting state if $\beta-1$ and
  $a-1$ share a common factor.
\end{proposition}

\begin{proof}
Suppose the expansion of $a+b$ in base $\beta$ is
\[
   a + b = n_r \beta^r + n_{r-1} \beta^{r-1} + \cdots + n_1 \beta + n_0.
\]
If $(a,b)$ is a reversed sum-product pair, then the expansion of $ab$ must be
\[
   ab = n_0 \beta^r + n_1 \beta^{r-1} + \cdots + n_{r-1} \beta + n_r.
\]
Reducing modulo $\beta - 1$ gives 
\[
  a + b  \equiv n_r + n_{r-1} + \cdots + n_1 + n_0 \equiv ab \pmod{\beta-1}.
\]
Massaging this congruence, we find that 
\[
  (a-1)(b-1) \equiv 1 \pmod{\beta-1}.  \qedhere
\]
\end{proof}


\section{An Opposite Problem}
\label{sec:participate}

Suppose that $(a,b)$ is a base-$\beta$ reversed sum-product pair with $b >
\beta$; this hypothesis rules out the uninteresting pairs $(2,2)$ and
$(\beta-1,\beta-1)$. We will say that $a$ \textbf{participates in a reversed
  sum-product pair} for the base $\beta$, or, for brevity,that $a$
\textbf{participates} for $\beta$. So $a=2$ and $a=3$ participate for the
base~10, but $a = 4$ and $a = 9$ do not.

\begin{example}
  \label{ex:participate}
For a given $a \ge 2$, we claim there are infinitely many bases $\beta$ for which $a$
participates. To see it, let $T$ be a variable. Define
\[
  b = (T+1)\beta + (aT+1) \qquad \text{and} \qquad \beta = (a^2-1)T + a-1.
\]
Then $(a,b)$ is a formal reversed sum-product pair for the base $\beta$ in the
sense that
\[
  a + b = (T+1)\beta + (aT + a + 1) \qquad \text{and} \qquad
  ab = (aT+a+1)\beta + (T+1).
\]
By setting $T$ equal to a positive integer, we get a reversed sum-product pair
in the usual sense unless $a = 2$ and $T = 1$ or $2$.
\end{example}

In the above example, we produced a set of bases $\beta$ for which $a$
participates; note that they all lie in the same congruence class modulo $a^2 -
1$. This is a general phenomenon: if $a$ participates for one base in a
congruence class, then it participates for all larger bases in the same class.

\begin{lemma}
  \label{lem:next_one}
  Let $a \ge 2$ be an integer, and suppose that $a$ participates for the base
  $\beta > a$. Then $a$ also participates for the base $\hat \beta = \beta +
  a^2 - 1$.
\end{lemma}

\begin{proof}
  Let $b = b_r \beta^r + \cdots + b_0$ be such that $(a,b)$ is a reversed
  sum-product pair for the base $\beta$. We will take the corresponding
  solutions to the recursion equations \eqref{eq:recursion_start} and
  \eqref{eq:recursion} and construct new solutions for the base $\hat \beta$.

  The quadruple $(b_r, b_0, \lambda, \rho)$ satisfies the equations
  \eqref{eq:recursion_start}. Define
  \[
     \hat b_r = b_r + \rho \qquad \hat b_0 = b_0 + a \rho.
   \]
  Since $\rho < a$, we see that $\hat b_r < \hat \beta$ and $\hat b_0 < \hat
  \beta - a$. One verifies immediately that the quadruple $(\hat b_r, \hat b_0,
  \lambda, \rho)$ satisfies the equations \eqref{eq:recursion_start} with $\hat
  \beta$ in place of $\beta$. Note that the carries $\lambda, \rho$ did not
  change.

  Now suppose that the quadruple $(x',y',\lambda',\rho')$ satisfies
  \eqref{eq:recursion}. Define
  \[
    \hat x' = x' + \lambda a + \rho' \qquad \hat y' = y' + \rho' a + \lambda.
    \]
  Again, we see that $\hat x', \hat y' < \hat \beta$ and the quadruple $(\hat
  x', \hat y', \lambda', \rho')$ satisfies \eqref{eq:recursion} with $\hat
  \beta$ in place of $\beta$.

  To complete the proof, we will show that the termination conditions agree;
  that is, we get an integer $\hat b$ with the same number of digits as $b$, and
  such that $(a,\hat b)$ is a reversed sum-product pair for the base $\hat
  \beta$. If the recursion for $b$ terminates because $\lambda = \rho$, then
  clearly the same is true for $\hat b$ since we used all of the same
  carries. If instead the recursion terminates because $x' = y'$ at some stage,
  then we claim that $\hat x' = \hat y'$. Indeed, we saw in
  \S\ref{sec:recursion} that $x' = y'$ implies that $\lambda = \rho'$ and
  $\lambda' = \rho$. It follows that
  \[
    \hat x' = x' + \lambda a + \rho' = y' + \rho' a + \lambda = \hat y'.
  \]
  It follows that the recursion terminates for $\hat b$, as desired. 
\end{proof}

\bigskip

There is at least one congruence class that contains no base for which $a$
participates.

\bigskip

\begin{lemma}
  \label{lem:no_zero}
  Let $(a,b)$ be a reversed sum-product pair for the base
  $\beta$ with $b > \beta$. Then $\beta \not\equiv 0 \pmod{a^2-1}$.
\end{lemma}

\medskip

\begin{proof}
  We can write $b = b_r \beta^r + \cdots + b_0$, with $r \ge 1$. Consequently,
  $b_r$ and $b_0$ must satisfy the equations \eqref{eq:recursion_start} for some
  choice of $\lambda, \rho$. Eliminating $b_0$ from \eqref{eq:recursion_start}
  shows that
  \[
     (a^2 - 1)b_r = a^2 - a\lambda + \rho \beta.
  \]
  If $\beta \equiv 0 \pmod{a^2-1}$, then reducing this equation modulo $a^2 - 1$
  yields $a^2 \equiv a \lambda \pmod{a^2-1}$. As $a$ is coprime to $a^2 - 1$, we
  conclude that $\lambda \equiv a \pmod{a^2-1}$. Since $a < a^2 - 1$, we must
  have $\lambda = a$, a contradiction.
\end{proof}

\bigskip

For integers $q,r$, let us agree to write $q\NN + r$ for the arithmetic
progression $\{qn + r \ : \ n = 0, 1, 2, \ldots \}$. 

\begin{theorem}
  \label{thm:participate}
  For $a \ge 2$, there exists a nonempty set of arithmetic progressions
  \[
    (a^2-1)\NN + v_1, \ (a^2-1)\NN + v_2, \ \ldots, \ (a^2-1)\NN + v_\ell
  \]
  such that
  \begin{itemize}
  \item $0 < v_i < a^2 - 1$ for all $i$;
    \smallskip
  \item If $a$ participates for $\beta$, then $\beta \in (a^2-1)\NN +
    v_i$ for some $i$; and
    \smallskip
  \item For each $i$ and each sufficiently large $\beta \in (a^2-1)\NN + v_i$,
    $a$ participates for $\beta$.
  \end{itemize}
\end{theorem}

\begin{proof}
  Fix $a \ge 2$ and let $B$ be the set of all $\beta \ge 2$ for which $a$
  participates. Define $v_1, \ldots, v_\ell$ to be the set of integers in the
  interval $[0, a^2-1)$ that are congruent to some element of
    $B$. Example~\ref{ex:participate} shows that the set $\{v_1, \ldots,
    v_\ell\}$ is nonempty. Lemma~\ref{lem:no_zero} shows that no $v_i = 0$. The
    final statement is immediate from Lemma~\ref{lem:next_one}.
\end{proof}

We can now address the question, ``How big is the set of bases for which a given
$a$ participates in a reversed sum-product pair?'' To that end, define the limit
\[
  \Omega(a) = \lim_{N \to \infty} \frac{1}{N}\left|\{ 2 \le \beta \le N \ : \ a \text{ participates for } \beta\}\right|.
\]

\begin{corollary}
  \label{cor:density}
  For each $a \ge 2$, the limit defining $\Omega(a)$ exists and is of the form
  $\ell / (a^2-1)$ for some integer $1 \le \ell < a^2 - 1$ that depends on
  $a$. In particular, $0 < \Omega(a) < 1$.
\end{corollary}

\begin{proof}
  Let $v_1, \ldots, v_\ell$ be as in the theorem. Then $1 \le \ell < a^2 -
  1$. For $N$ sufficiently large, we have
  \begin{align*}
    &\left|\{ 2 \le \beta \le N \ : \ a \text{ participates for } \beta\}\right| \\
    &\phantom{xxxxxx}= \left| \bigcup_{i=1}^\ell \left\{(a^2-1)n + v_i \ : \ 1 \le n \le \frac{N}{a^2-1}\right\}\right| + O(1) \\
    &\phantom{xxxxxx}= \sum_{i=1}^\ell \frac{N}{a^2-1} + O(1) 
    = \frac{\ell N}{a^2-1} + O(1).
  \end{align*}
Dividing by $N$ and passing to the limit gives the result. 
\end{proof}

\begin{example}
  \label{ex:Omega2}
  We claim that $\Omega(2) = \frac{1}{3}$. The theorem shows that all bases
  $\beta$ for which $a=2$ participates lies in the arithmetic progressions $3\ZZ
  + 1$ or $3\ZZ + 2$. We will now argue that if $\beta \in 3\ZZ + 2$, then there
  is no reversed sum-product pair for the base $\beta$ other than the
  uninteresting pair $(2,2)$. Write $\beta = 3n + 2$. Solving the recursion
  equations \eqref{eq:recursion_start} for $b_r$, shows that $3b_r = 4 -
  2\lambda + \rho\beta$. Reducing modulo 3 and simplifying gives $\rho \equiv
  \lambda + 1 \pmod{3}$. The only solution to this congruence with $0 \le
  \lambda, \rho < 2$ is $\lambda = 0$ and $\rho = 1$. As $\lambda \ne \rho$, we
  must continue the recursion. A similar analysis applied to
  \eqref{eq:recursion} shows that $\rho' = 0$ and $\lambda' = 1$. But then the
  equations \eqref{eq:recursion} become
  \begin{align*}
    2x' + 1 &= y' \\
    2y' + 1 &= x',
  \end{align*}
  which has $x' = y' = -1$ as their unique solution. These are not valid digits
  in any base. This contradiction shows that the only reversed sum-product
  pair for a base $\beta \in 3\ZZ + 2$ is $(2,2)$.
\end{example}

A similar strategy to the one in Example~\ref{ex:Omega2} allows us to compute
the ratio $\Omega(a)$ for any $a$.  Table~\ref{tab:omega} gives the first few
values. (Given $a$, the procedure for determining which arithmetic progressions
actually contain bases $\beta$ for which $a$ participates is implemented in the
function \texttt{construct\_generic\_automata} in our Python module.)

\begin{table}[!ht]
\begin{tabular}{rlcrl}\toprule
  $a$ & \hspace{.5cm} $\Omega(a)$ & \hspace{.5cm} & $a$ & \hspace{.5cm} $\Omega(a)$ \\
  \midrule
  2 & $ 1/3 \approx 0.333 $ && 7 & $ 4/48 \approx 0.083$\\
  3 & $ 1/8 = 0.125 $ && 8 & $ 22/63 \approx 0.349$\\
  4 & $ 4/15 \approx 0.267 $ &&  9 & $ 12/80 = 0.15 $\\
  5 & $ 3/24 = 0.125 $ && 10 & $ 26/99 \approx 0.263 $ \\
  6 & $ 13/35 \approx 0.371 $\\
  \bottomrule
\end{tabular}
\vspace{.2cm}
\caption{The value of the ratio $\Omega(a)$ for $a \le 10$.}
\label{tab:omega}
\end{table}


\section{Existence of Interesting Pairs}
\label{sec:compute}

Recall that a reversed sum-product pair $(a,b)$ for the base $\beta$ is deemed
to be \textbf{interesting} if it is not one of the pairs $(2,2)$ or
$(\beta-1,\beta-1)$.  For a given base $\beta$, do we expect to find any
interesting reversed sum-product pair at all? The following propositions provide
support for Conjecture~\ref{conj:interesting} in the introduction.

\begin{proposition}
  \label{prop:few-exceptions1}
  Among all bases $\beta \ge 2$, at least $99.3\%$ of them admit an interesting
  reversed sum product-pair.
\end{proposition}

\begin{proposition}
  \label{prop:few-exceptions2}
  The only bases $\beta < 1,441,440$ for which there is no interesting reversed
  sum-product pair are
  \[
  2,3,4,5,6,7,8,9,12,15,21.
  \]
\end{proposition}

We describe a sieving procedure that will allow us to prove both propositions
simultaneously:
\begin{enumerate}
\item Set $B 
  = 1,441,440$,
  and consider an array of integers from $0$ to $B-1$.
\item For each $a < 100$ such that $a^2 - 1$ divides $B$, do the following:
  
    \begin{itemize}
      \item[(a)] Compute $v_1, \ldots, v_\ell$ as described by Theorem~\ref{thm:participate}.
      \item[(b)] For each arithmetic progression $(a^2-1)\NN + v_i$, let $T_i
        \ge 0$ be minimal such that $\beta = (a^2-1)T_i + v_i$ admits  an
        interesting reversed sum-product pair.
      \item[(c)] For each $t \ge T_i$ such that $ \beta = (a^2-1)t + v_i$ lies in the
        interval $[0, B)$, cross $\beta$ off of our array.
    \end{itemize}
  \item For each $\beta \in [0,B)$ that we have not crossed off yet, construct
    the DFAs $A_{\beta,a}$ for $a \ge 2$. If we find one that gives rise to an
    interesting reversed sum-product pair, cross $\beta$ off of our array.
\end{enumerate}

Step~(2a) can be accomplished by a ``generic'' version of the construction in
\S\ref{sec:DFA}. Set $\beta = (a^2-1)T + v$ for some fixed $0 < v < a^2-1$. We
solve the recursion equations~\eqref{eq:recursion_start} for integers $0 \le
\lambda, \rho < a$ and with $b_0,b_r$ being linear polynomials in $T$. Then we
solve \eqref{eq:recursion} for $0 \le \lambda',\rho' < a$ and with $x',y'$ being
linear polynomials in $T$. This is implemented in our Python module in the
function \texttt{construct\_generic\_automata}.

To determine $T_i$ as in Step~(2b), we can specialize a generic DFA to $T = 0,
1, \ldots$ until we find a valid DFA $A_{(a^2-1)T+v_i, a}$.  At the end of
Step~(2), we find that $1,431,542$ of the elements of our array have been
crossed off. Since $a^2-1$ divides $B$ for each $a$ used in the computation,
Lemma~\ref{lem:next_one} tells us that all sufficiently large integers $\beta$
in each of these residue classes admit an interesting reversed sum-product
pair. That is, this holds for at least $\frac{1,431,542}{1,441,440} \approx
99.31\%$ of all bases, which proves Proposition~\ref{prop:few-exceptions1}. This
part of the computation took approximately 5.5 minutes on a Xeon(R) E5-2699
processor (2.30GHz with 500GB memory).

The construction of the DFAs $A_{\beta,a}$ as in Step~(3) is implemented in our
Python module in the function \texttt{construct\_automata}. Applying it, we
cross off all remaining entries in the array except for
\[
0, 1, 2,3,4,5,6,7,8,9,12,15,21.
\]
Since $\beta = 0,1$ are not valid numerical bases, we may drop these from
consideration, thus proving Proposition~\ref{prop:few-exceptions2}. This step
required an additional 1.5 hours to complete on the same processor as above.

\begin{remark}
Trying Step~(2) again with the larger parameter $B = 53,542,288,800$ took $42$
hours and gives the improved lower bound of $\frac{53497192379}{53542288800}
\approx 99.92\%$ in Proposition~\ref{prop:few-exceptions1}. We did not extend
the computation in Step~(3) to this larger value of $B$.
\end{remark}

\medskip

\noindent \textbf{Acknowledgments}. The first author would like to thank his
children, Henry and Yona, for the conversation that stimulated this work, and
for their patience with him as a mathematics teacher during a worldwide
pandemic. The second author would like to thank his children, Jack and Salem,
for their patience with him as an officemate for the first few months of the
pandemic. We also thank one of the anonymous referees for pointing us toward
related literature.

\bibliographystyle{plain}
\bibliography{rsp}
\end{document}